\documentclass[a4paper,10pt]{amsart}

\usepackage[utf8]{inputenc}
\usepackage{url}

\usepackage[expansion=false]{microtype}
\usepackage{amssymb,amsmath,amsopn,amsxtra,amsthm,amsfonts}
\usepackage{xr}
\usepackage[mathcal]{euscript}
\usepackage{mathalfa}

\usepackage[english, status=final, nomargin, inline]{fixme}
\fxusetheme{color}

\usepackage[pdfusetitle,unicode,hidelinks]{hyperref}

\usepackage{enumitem}
\setlist[enumerate,1]{label={(\arabic*)},itemsep=\parskip} %,leftmargin=0pt
\setlist[itemize,1]{itemsep=\parskip} %,leftmargin=0pt
\newlist{thmlist}{enumerate}{2}
\setlist[thmlist,1]{label={\em(\roman*)},ref={(\roman*)},%
  itemsep=\parskip,leftmargin=*,align=left}
\setlist[thmlist,2]{label={\em(\alph*)},ref={(\alph*)},%
  itemsep=\parskip,leftmargin=*,align=left,topsep=0.1cm}
\newlist{remlist}{enumerate}{2}
\setlist[remlist,1]{label={(\roman*)},ref={(\roman*)},itemsep=\parskip,%
  leftmargin=*,align=left}
\setlist[remlist,2]{label={(\alph*)},ref={(\alph*)},itemsep=\parskip,%
  leftmargin=*,align=left,topsep=0.1cm}

\makeatletter
\let\c@equation\c@subsubsection

\makeatother

\newtheorem{cor}[subsubsection]{Corollary}
\newtheorem{lem}[subsubsection]{Lemma}

\newtheorem{thm}[subsubsection]{Theorem}

\newtheorem{claim}[subsubsection]{Claim}
\newtheorem*{claim*}{Claim}

\theoremstyle{definition}

\newtheorem{rem}[subsubsection]{Remark}

\renewcommand{\eqref}[1]{(\ref{#1})}

\usepackage{tikz}
\usetikzlibrary{matrix}
\usepackage{tikz-cd}

\usepackage{xspace}

\usepackage{xifthen}

\usepackage{xparse}

\newcommand{\nc}{\newcommand}
\nc{\renc}{\renewcommand}
\nc{\ssec}{\subsection}
\nc{\sssec}{\subsubsection}
\nc{\on}{\operatorname}
\nc{\term}[1]{#1\xspace}

\DeclareMathSymbol{A}{\mathalpha}{operators}{`A}
\DeclareMathSymbol{B}{\mathalpha}{operators}{`B}
\DeclareMathSymbol{C}{\mathalpha}{operators}{`C}
\DeclareMathSymbol{D}{\mathalpha}{operators}{`D}
\DeclareMathSymbol{E}{\mathalpha}{operators}{`E}
\DeclareMathSymbol{F}{\mathalpha}{operators}{`F}
\DeclareMathSymbol{G}{\mathalpha}{operators}{`G}
\DeclareMathSymbol{H}{\mathalpha}{operators}{`H}
\DeclareMathSymbol{I}{\mathalpha}{operators}{`I}
\DeclareMathSymbol{J}{\mathalpha}{operators}{`J}
\DeclareMathSymbol{K}{\mathalpha}{operators}{`K}
\DeclareMathSymbol{L}{\mathalpha}{operators}{`L}
\DeclareMathSymbol{M}{\mathalpha}{operators}{`M}
\DeclareMathSymbol{N}{\mathalpha}{operators}{`N}
\DeclareMathSymbol{O}{\mathalpha}{operators}{`O}
\DeclareMathSymbol{P}{\mathalpha}{operators}{`P}
\DeclareMathSymbol{Q}{\mathalpha}{operators}{`Q}
\DeclareMathSymbol{R}{\mathalpha}{operators}{`R}
\DeclareMathSymbol{S}{\mathalpha}{operators}{`S}
\DeclareMathSymbol{T}{\mathalpha}{operators}{`T}
\DeclareMathSymbol{U}{\mathalpha}{operators}{`U}
\DeclareMathSymbol{V}{\mathalpha}{operators}{`V}
\DeclareMathSymbol{W}{\mathalpha}{operators}{`W}
\DeclareMathSymbol{X}{\mathalpha}{operators}{`X}
\DeclareMathSymbol{Y}{\mathalpha}{operators}{`Y}
\DeclareMathSymbol{Z}{\mathalpha}{operators}{`Z}

\nc{\sA}{\ensuremath{\mathcal{A}}\xspace}
\nc{\sB}{\ensuremath{\mathcal{B}}\xspace}
\nc{\sC}{\ensuremath{\mathcal{C}}\xspace}
\nc{\sD}{\ensuremath{\mathcal{D}}\xspace}
\nc{\sE}{\ensuremath{\mathcal{E}}\xspace}
\nc{\sF}{\ensuremath{\mathcal{F}}\xspace}
\nc{\sG}{\ensuremath{\mathcal{G}}\xspace}
\nc{\sH}{\ensuremath{\mathcal{H}}\xspace}
\nc{\sI}{\ensuremath{\mathcal{I}}\xspace}
\nc{\sJ}{\ensuremath{\mathcal{J}}\xspace}
\nc{\sK}{\ensuremath{\mathcal{K}}\xspace}
\nc{\sL}{\ensuremath{\mathcal{L}}\xspace}
\nc{\sM}{\ensuremath{\mathcal{M}}\xspace}
\nc{\sN}{\ensuremath{\mathcal{N}}\xspace}
\nc{\sO}{\ensuremath{\mathcal{O}}\xspace}
\nc{\sP}{\ensuremath{\mathcal{P}}\xspace}
\nc{\sQ}{\ensuremath{\mathcal{Q}}\xspace}
\nc{\sR}{\ensuremath{\mathcal{R}}\xspace}
\nc{\sS}{\ensuremath{\mathcal{S}}\xspace}
\nc{\sT}{\ensuremath{\mathcal{T}}\xspace}
\nc{\sU}{\ensuremath{\mathcal{U}}\xspace}
\nc{\sV}{\ensuremath{\mathcal{V}}\xspace}
\nc{\sW}{\ensuremath{\mathcal{W}}\xspace}
\nc{\sX}{\ensuremath{\mathcal{X}}\xspace}
\nc{\sY}{\ensuremath{\mathcal{Y}}\xspace}
\nc{\sZ}{\ensuremath{\mathcal{Z}}\xspace}

\nc{\bA}{\ensuremath{\mathbf{A}}\xspace}
\nc{\bB}{\ensuremath{\mathbf{B}}\xspace}
\nc{\bC}{\ensuremath{\mathbf{C}}\xspace}
\nc{\bD}{\ensuremath{\mathbf{D}}\xspace}
\nc{\bE}{\ensuremath{\mathbf{E}}\xspace}
\nc{\bF}{\ensuremath{\mathbf{F}}\xspace}
\nc{\bG}{\ensuremath{\mathbf{G}}\xspace}
\nc{\bH}{\ensuremath{\mathbf{H}}\xspace}
\nc{\bI}{\ensuremath{\mathbf{I}}\xspace}
\nc{\bJ}{\ensuremath{\mathbf{J}}\xspace}
\nc{\bK}{\ensuremath{\mathbf{K}}\xspace}
\nc{\bL}{\ensuremath{\mathbf{L}}\xspace}
\nc{\bM}{\ensuremath{\mathbf{M}}\xspace}
\nc{\bN}{\ensuremath{\mathbf{N}}\xspace}
\nc{\bO}{\ensuremath{\mathbf{O}}\xspace}
\nc{\bP}{\ensuremath{\mathbf{P}}\xspace}
\nc{\bQ}{\ensuremath{\mathbf{Q}}\xspace}
\nc{\bR}{\ensuremath{\mathbf{R}}\xspace}
\nc{\bS}{\ensuremath{\mathbf{S}}\xspace}
\nc{\bT}{\ensuremath{\mathbf{T}}\xspace}
\nc{\bU}{\ensuremath{\mathbf{U}}\xspace}
\nc{\bV}{\ensuremath{\mathbf{V}}\xspace}
\nc{\bW}{\ensuremath{\mathbf{W}}\xspace}
\nc{\bX}{\ensuremath{\mathbf{X}}\xspace}
\nc{\bY}{\ensuremath{\mathbf{Y}}\xspace}
\nc{\bZ}{\ensuremath{\mathbf{Z}}\xspace}

\nc{\dA}{\ensuremath{\mathds{A}}\xspace}
\nc{\dB}{\ensuremath{\mathds{B}}\xspace}
\nc{\dC}{\ensuremath{\mathds{C}}\xspace}
\nc{\dD}{\ensuremath{\mathds{D}}\xspace}
\nc{\dE}{\ensuremath{\mathds{E}}\xspace}
\nc{\dF}{\ensuremath{\mathds{F}}\xspace}
\nc{\dG}{\ensuremath{\mathds{G}}\xspace}
\nc{\dH}{\ensuremath{\mathds{H}}\xspace}
\nc{\dI}{\ensuremath{\mathds{I}}\xspace}
\nc{\dJ}{\ensuremath{\mathds{J}}\xspace}
\nc{\dK}{\ensuremath{\mathds{K}}\xspace}
\nc{\dL}{\ensuremath{\mathds{L}}\xspace}
\nc{\dM}{\ensuremath{\mathds{M}}\xspace}
\nc{\dN}{\ensuremath{\mathds{N}}\xspace}
\nc{\dO}{\ensuremath{\mathds{O}}\xspace}
\nc{\dP}{\ensuremath{\mathds{P}}\xspace}
\nc{\dQ}{\ensuremath{\mathds{Q}}\xspace}
\nc{\dR}{\ensuremath{\mathds{R}}\xspace}
\nc{\dS}{\ensuremath{\mathds{S}}\xspace}
\nc{\dT}{\ensuremath{\mathds{T}}\xspace}
\nc{\dU}{\ensuremath{\mathds{U}}\xspace}
\nc{\dV}{\ensuremath{\mathds{V}}\xspace}
\nc{\dW}{\ensuremath{\mathds{W}}\xspace}
\nc{\dX}{\ensuremath{\mathds{X}}\xspace}
\nc{\dY}{\ensuremath{\mathds{Y}}\xspace}
\nc{\dZ}{\ensuremath{\mathds{Z}}\xspace}

\nc{\bbA}{\ensuremath{\mathbb{A}}\xspace}
\nc{\bbB}{\ensuremath{\mathbb{B}}\xspace}
\nc{\bbC}{\ensuremath{\mathbb{C}}\xspace}
\nc{\bbD}{\ensuremath{\mathbb{D}}\xspace}
\nc{\bbE}{\ensuremath{\mathbb{E}}\xspace}
\nc{\bbF}{\ensuremath{\mathbb{F}}\xspace}
\nc{\bbG}{\ensuremath{\mathbb{G}}\xspace}
\nc{\bbH}{\ensuremath{\mathbb{H}}\xspace}
\nc{\bbI}{\ensuremath{\mathbb{I}}\xspace}
\nc{\bbJ}{\ensuremath{\mathbb{J}}\xspace}
\nc{\bbK}{\ensuremath{\mathbb{K}}\xspace}
\nc{\bbL}{\ensuremath{\mathbb{L}}\xspace}
\nc{\bbM}{\ensuremath{\mathbb{M}}\xspace}
\nc{\bbN}{\ensuremath{\mathbb{N}}\xspace}
\nc{\bbO}{\ensuremath{\mathbb{O}}\xspace}
\nc{\bbP}{\ensuremath{\mathbb{P}}\xspace}
\nc{\bbQ}{\ensuremath{\mathbb{Q}}\xspace}
\nc{\bbR}{\ensuremath{\mathbb{R}}\xspace}
\nc{\bbS}{\ensuremath{\mathbb{S}}\xspace}
\nc{\bbT}{\ensuremath{\mathbb{T}}\xspace}
\nc{\bbU}{\ensuremath{\mathbb{U}}\xspace}
\nc{\bbV}{\ensuremath{\mathbb{V}}\xspace}
\nc{\bbW}{\ensuremath{\mathbb{W}}\xspace}
\nc{\bbX}{\ensuremath{\mathbb{X}}\xspace}
\nc{\bbY}{\ensuremath{\mathbb{Y}}\xspace}
\nc{\bbZ}{\ensuremath{\mathbb{Z}}\xspace}

\nc{\mrm}[1]{\ensuremath{\mathrm{#1}}\xspace}
\nc{\mbf}[1]{\ensuremath{\mathbf{#1}}\xspace}
\nc{\mcal}[1]{\ensuremath{\mathcal{#1}}\xspace}
\nc{\msc}[1]{\ensuremath{\mathscr{#1}}\xspace}

\renc{\bar}[1]{\overline{#1}}

\nc{\sub}{\subset}
\nc{\too}{\longrightarrow}
\nc{\hook}{\hookrightarrow}
\nc*{\hooklongrightarrow}{\ensuremath{\lhook\joinrel\relbar\joinrel\rightarrow}}
\nc{\hooklong}{\hooklongrightarrow}
\nc{\twoheadlongrightarrow}{\relbar\joinrel\twoheadrightarrow}
\nc{\shiso}{\approx}
\nc{\isoto}{\xrightarrow{\sim}}
\nc{\isofrom}{\xleftarrow{\sim}}
\renc{\ge}{\geqslant}
\renc{\le}{\leqslant}
\renc{\geq}{\geqslant}
\renc{\leq}{\leqslant}

\nc{\id}{\mathrm{id}}

\DeclareMathOperator{\Hom}{\mathrm{Hom}}
\nc{\uHom}{\underline{\smash{\Hom}}}

\DeclareMathOperator{\End}{\mathrm{End}}

\nc{\Pre}{\mathrm{PSh}{}}
\nc{\Shv}{\mathrm{Shv}{}}
\nc{\uEnd}{\underline{\smash{\End}}}

\renc{\lim}{\operatorname*{lim}}
\nc{\colim}{\operatorname*{colim}}
\nc{\Cofib}{\on{Cofib}}
\nc{\Fib}{\on{Fib}}
\nc{\initial}{\varnothing}
\nc{\op}{\mathrm{op}}

\renc{\coprod}{\sqcup}

\nc{\bDelta}{\mbf{\Delta}}
\nc{\DM}{\mbf{DM}}
\nc{\eff}{\mathrm{eff}}
\nc{\veff}{\mathrm{veff}}
\nc{\cyc}{{\mrm{cyc}}}
\nc{\corr}{{\on{corr}}}
\nc{\ft}{\mrm{ft}}
\nc{\flf}{\mrm{flf}}
\nc{\fet}{{\mrm{f\acute et}}}
\nc{\fsyn}{{\mrm{fsyn}}}
\nc{\syn}{{\mrm{syn}}}
\nc{\lci}{{\mrm{lci}}}
\nc{\Perf}{\mbf{Perf}}
\nc{\perf}{\mrm{perf}}
\nc{\oblv}{\mrm{oblv}}
\nc{\exact}{\on{exact}}

\nc{\F}{{\on{F}}}
\nc{\clopen}{{\mrm{clopen}}}
\nc{\B}{\mrm{B}}
\nc{\D}{\mrm{D}}
\nc{\Fin}{\on{Fin}}
\nc{\fin}{\mrm{fin}}
\nc{\Cut}{\on{Cut}}
\nc{\Cart}{\on{Cart}}
\nc{\pairs}{\mathsf{pairs}}
\nc{\Pairs}{\mathrm{Pair}}
\nc{\Trip}{\mathrm{Trip}}
\nc{\Lab}{\mathrm{Lab}}
\nc{\SL}{\mathrm{SL}}
\nc{\coCart}{\mathrm{coCart}}
\nc{\RKE}{\mathrm{RKE}}
\nc{\strict}{\mathrm{strict}}
\nc{\Emb}{\mathrm{Emb}}
\nc{\Split}{\mathrm{Split}}
\nc{\Set}{\mathrm{Set}}
\nc{\sSets}{\mathrm{sSets}}
\nc{\pb}{\mathrm{pb}}
\nc{\fib}{\mathrm{fib}}
\nc{\diff}{\mrm{diff}}
\nc{\gp}{\mrm{gp}}
\nc{\map}{\mrm{map}}

\nc{\mgp}{\mrm{mot-gp}}
\nc{\FSyn}{\mrm{FSyn}}
\nc{\FEt}{\mrm{FEt}}
\nc{\Spc}{\mrm{Spc}}
\nc{\Ob}{\mrm{Ob}}
\nc{\Spt}{\mrm{Spt}}
\nc{\T}{\bT}
\nc{\suspinf}{\Sigma^\infty}
\nc{\h}{\mrm{h}}
\nc{\uhom}{\underline{\mathrm{Hom}}}
\nc{\umap}{\underline{\mathrm{Maps}}}
\renc{\H}{\bH}
\nc{\Einfty}{{\sE_\infty}}
\nc{\Eone}{{\sE_1}}
\nc{\Stab}{\mrm{Stab}}
\nc{\lax}{{\mrm{lax}}}
\nc{\cocart}{{\mrm{cocart}}}
\nc{\Sch}{\on{Sch}}
\nc{\Fr}{\on{Fr}}
\nc{\A}{\mathbf{A}}
\nc{\N}{\mathbf{N}}
\nc{\Z}{\mathbf{Z}}
\nc{\Q}{\mathbf{Q}}
\nc{\Oo}{\mathcal{O}} 
\nc{\Fscr}{\mathcal{F}}
\nc{\Gscr}{\mathcal{G}}
\nc{\Ll}{\mathcal{L}} 
\nc{\Mm}{\mathcal{M}} 
\nc{\mm}{\mathrm{m}} 
\nc{\K}{\mrm{K}} 
\nc{\W}{\mrm{W}} 
\nc{\red}{{\on{red}}}
\nc{\Voev}{{\on{Voev}}}
\nc{\Corr}{\mrm{Corr}}
\nc{\Span}{\mathbf{Corr}}
\nc{\Gap}{\mrm{Gap}}
\nc{\Corrfr}{\Corr^{\fr}}
\nc{\Corrvfr}{\Corr^{\Vfr}}
\nc{\Spec}{\on{Spec}}
\nc{\Sm}{\on{Sm}}
\nc{\Gm}{\mathbf{G}_{\on{m}}}
\renc{\P}{\bP}
\nc{\nis}{\mathrm{nis}}
\nc{\zar}{\mathrm{zar}}
\nc{\et}{\mathrm{\acute et}}
\nc{\all}{\mathrm{all}}
\nc{\fold}{\mathrm{fold}}
\nc{\Fun}{\mathrm{Fun}}
\nc{\Ho}{\mathrm{Ho}}
\nc{\Segal}{\mathrm{Segal}}
\nc{\Mon}{\mrm{Mon}{}}
\nc{\Ab}{\mrm{Ab}}
\nc{\Sh}{\on{Sh}}
\nc{\M}{\mrm{M}}
\nc{\Lhtp}{L_{\A^1}}
\nc{\Lmot}{L_{\mrm{mot}}}
\nc{\mot}{\mrm{mot}}
\nc{\SH}{\mbf{SH}}
\nc{\RR}{\mbf{R}}
\nc{\CC}{\mbf{C}}
\nc{\Mod}{\mbf{Mod}}
\nc{\QCoh}{\mbf{QCoh}}
\nc{\MonUnit}{\mbf{1}}
\nc{\1}{\mbf{1}}
\nc{\tr}{\on{tr}}
\nc{\cotr}{\mrm{cotr}}
\nc{\vop}{\mrm{vop}}
\nc{\fr}{{\on{fr}}}
\nc{\Ar}{\mrm{Ar}}
\nc{\Vfr}{\on{Vfr}}
\nc{\frdiff}{{\on{frdiff}}}
\nc{\frGys}{\on{frGys}}
\nc{\SHfr}{\SH^{\fr}}
\nc{\SHfrdiff}{\SH^{\frdiff}}
\nc{\SHfrGys}{\SH^{\frGys}}
\nc{\InftyCat}{(\mathrm{\infty,1)\textnormal{-}Cat}}
\nc{\TriCat}{\mathrm{TriCat}}
\nc{\oneCat}{\mathrm{1\textnormal{-}Cat}}
\nc{\Cat}{\mathrm{Cat}}
\nc{\Th}{\on{Th}}

\nc{\CMon}{\mrm{CMon}{}}
\nc{\CAlg}{\mrm{CAlg}{}}
\nc{\MGL}{\mrm{MGL}}
\nc{\Seg}{\mrm{Seg}{}}
\nc{\GW}{\mrm{GW}{}}
\nc{\Tw}{\mrm{Tw}}
\nc{\sslash}{/\mkern-6mu/}
\nc{\PrL}{\mrm{Pr}^\mrm{L}}
\nc{\PrR}{\mrm{Pr}^\mrm{R}}
\nc{\pr}{\mrm{pr}}
\let\phi\varphi

\nc\efr{\mrm{efr}}
\nc\nfr{\mrm{nfr}}
\nc\dfr{\mrm{fr}}
\nc\tfr{\mrm{tfr}}
\nc\Vect{\mrm{Vect}}
\nc\sVect{\mrm{sVect}}
\nc{\fix}{\mrm{fix}}
\nc{\ho}{\mrm{h}}
\nc\Mfd{\mrm{Mfd}}
\nc{\PSh}{\mrm{PSh}}
\nc{\hzmw}{H \tilde\Z{}}
\nc{\Cor}{\mrm{Cor}{}}
\nc{\cormw}{\mrm{\widetilde{Cor}}{}}
\nc{\Chw}{\mrm{\widetilde{CH}}{}}
\nc{\Ex}{\mrm{Ex}}
\nc{\BM}{\mrm{BM}}

\nc{\Pic}{\mrm{Pic}}
\nc{\Br}{\mrm{Br}}
\nc{\pur}{\mathfrak p}
\nc{\angles}[1]{\langle #1\rangle}
\nc{\inv}[1]{[\tfrac{1}{#1}]}
\nc{\pinv}{\inv{p}}
\nc{\cinv}{\inv{p}}
\nc{\Sph}{\on{Sph}}
\nc{\KGL}{\mrm{KGL}}
\nc{\KH}{\mrm{KH}}
\nc{\Flag}{\mrm{Flag}}
\nc{\Pro}{\mrm{Pro}}
\nc{\Frac}{\mrm{Frac}}

\nc{\arc}{\mrm{arc}}
\nc{\rarc}{\mrm{rarc}}
\nc{\cdarc}{\mrm{cdarc}}
\nc{\vv}{\mrm{v}}
\nc{\rv}{\mrm{rv}}
\nc{\cdv}{\mrm{cdv}}
\nc{\hh}{\mrm{h}}
\nc{\cdh}{\mrm{cdh}}
\nc{\rh}{\mathrm{rh}}
\nc{\Et}{\mathrm{Et}}
\nc{\Nis}{\mathrm{Nis}}
\nc{\Zar}{\mathrm{Zar}}
\nc{\cdp}{\mathrm{cdp}}

\nc{\RZ}{\mathrm{RZ}}
\nc{\qcqs}{\mathrm{qcqs}}
\nc{\aff}{\mathrm{aff}}
\nc{\cl}{\mathrm{cl}}
\nc{\Val}{\mathrm{Val}}
\nc{\GFin}{\mathrm{GFin}{}}
\nc{\Proj}{\mathrm{Proj}}

\nc{\inftyCat}{\term{$\infty$-category}}
\nc{\inftyCats}{\term{$\infty$-categories}}

\nc{\inftyOneCat}{\term{$(\infty,1)$-category}}
\nc{\inftyOneCats}{\term{$(\infty,1)$-categories}}

\nc{\inftyGrpd}{\term{$\infty$-groupoid}}
\nc{\inftyGrpds}{\term{$\infty$-groupoids}}

\nc{\inftyTop}{\term{$\infty$-topos}}
\nc{\inftyTops}{\term{$\infty$-toposes}}

\nc{\inftyTwoCat}{\term{$(\infty,2)$-category}}
\nc{\inftyTwoCats}{\term{$(\infty,2)$-categories}}

\title{}
\title{A Descent view on Mitchell's theorem}

\author[E. Elmanto]{Elden Elmanto}
\address{Department of Mathematics\\
Harvard University\\
1 Oxford St.\
Cambridge, MA 02138\\
USA}
\email{\href{mailto:elmanto@math.harvard.edu}{elmanto@math.harvard.edu}}
\urladdr{\url{https://www.eldenelmanto.com/}}

\author[D. Nardin]{Denis Nardin}
\address{Fakult\"at f\"ur Mathematik\\
Universit\"at Regensburg\\
93040 Regensburg\\
Germany}
\email{\href{mailto:denis.nardin@ur.de}{denis.nardin@ur.de}}
\urladdr{\url{https://homepages.uni-regensburg.de/~nad22969/}}

\author[L. Yang]{Lucy Yang}
\address{Department of Mathematics\\
Harvard University\\
1 Oxford St.\
Cambridge, MA 02138\\
USA}
\email{\href{mailto:lyang@math.harvard.edu}{lyang@math.harvard.edu}}
\begin{document}

\def\comp{\wedge}

\maketitle

\begin{abstract} 
    In this short note, we give a new proof of Mitchell's theorem that $L_{T(n)}K(\bZ) \simeq 0$ for $n \geq 2$. 
    Instead of reducing the problem to delicate representation theory, we use recently established hyperdescent technology for chromatically-localized algebraic $K$-theory. 
    \end{abstract}

\section{Introduction and background}

 In this note, we give an alternate proof of the following result:

\begin{thm}[Mitchell]\label{thm:mitchell} For all primes $p$ and $n \geq 2$, $K(n)_*K(\bZ)  = 0$.
\end{thm}

The proof Theorem~\ref{thm:mitchell} given in \cite{Mit90}  is relatively self-contained and depends on showing that the unit map $\1 \rightarrow K(\bZ)$ factors through the ``image of $J$." This factoring relies on delicate representation theory by way of the Barratt-Priddy-Quillen theorem. Since the latter spectrum is known to be acyclic for the Morava's $K(n)$ for $n \geq 2$, the result follows. The value of the present note is that it locates the proof in its natural environment --- Rognes' redshift philosophy in algebraic $K$-theory. 

The starting point of Theorem~\ref{thm:mitchell} is Thomason's seminal result that $K(1)$-localized algebraic $K$-theory satisfies \'etale descent \cite{aktec}. Combined with the rigidity theorems of Suslin \cite{Sus83} and Gabber \cite{Gab92}, one concludes that $K(1)$-local algebraic $K$-theory is, more or less, topological $K$-theory; we also refer the reader to \cite{eventually} for further elaboration on this point of view.

One can view Thomason's theorem as a ``Bott-asymptotic" version of a more refined statement --- the Quillen-Lichtenbaum conjecture (now the Voevodsky-Rost theorem \cite{VV,Voevodsky:2008}) which asserts that algebraic and \'etale $K$-theory agrees in high enough degrees. By analogy with the Quillen-Lichtenbaum conjectures, Rognes has formulated the idea that taking algebraic $K$-theory increases ``chromatic complexity" --- demonstrating a ``redshift"; we refer the reader to \cite{rognes-msri} for a discussion.
  
In line with this ideology, Thomason's theorem is then viewed as saying that taking algebraic $K$-theory of a discrete commutative ring (which is $K(1)$-acyclic) yields an interesting answer $K(1)$-locally. At the next height, results of Ausoni-Rognes \cite{ausoni, ausoni-rognes} who confirmed that $v_2$ acts invertibly on $K(K(\bC)^{\comp}_p) \otimes V(1)$ where $V(1)$ is a type $2$ complex, hence is interesting $K(2)$-locally. 

We can view Mitchell's result anachronistically as a demonstration of the strictness of redshift: while the 2-fold algebraic $K$-theory of a discrete ring is interesting $K(2)$-locally, the algebraic $K$-theory thereof itself is not. 

The value of our proof, if there is one, is that it is born in the same spirit as Thomason's results: we confirm Mitchell's vanishing by way of \'etale hyperdescent. We believe that proving the result this way places it within its proper context, at the cost of using more technology. 

%Roughly speaking, we use recent results of Clausen-Mathew \cite{clausen-mathew} which is based on a reproof of Thomason's result by Clausen-Mathew-Noel-Naumann \cite{cmnn} to reduce to verifying the above result for separably close fields, where we can appeal to rigidity results of Suslin and Gabber. Just as Theorem~\ref{thm:mitchell} is a kind of height $n\geq 2$ version of Thomason's result, we view our proof as a height $n \geq 2$ version of Thomason's proof. 

%The following is a recent, spectacular result of Land-Meier-Tamme \cite{land-meier-tamme} located within Rognes' redshift philosophy in algebraic $K$-theory 

%\begin{thm}[Land-Meier-Tamme, Corollary 2.32]\label{thm:lmt} Let $A$ be a $\bE_{\infty}$-ring spectrum. If $A$ is $K(1)$-acyclic, then $K(A)$ is $K(n)$-acyclic as well for $n \geq 2$.
%\end{thm}
%
%The proof of Theorem~\ref{thm:lmt} goes something like this: \emph{op. cit.} establishes a general pattern that $T(n)$-local $K$-theory is truncating when restricted to $T(1) \oplus \cdots T(n)$

\subsection{Acknowledgements} We would like to thank Akhil Mathew and John Rognes for useful discussions and Gabriel Angelini-Knoll, Ben Antieau, Markus Land, Lennart Meier and Georg Tamme for comments on an earlier draft.
\section{On Mitchell's theorem}

\ssec{A $p$-adic version of Mitchell's theorem} 
As a warm-up, we first give a very simple proof of the $p$-adic version of Mitchell's theorem. So fix a prime $p > 0$; here our $T(n)$-localizations will be at this implicit prime.

\begin{thm} \label{thm:baby-mitchell} For $n \geq 2$, we have that $L_{T(n)}K\left(\bZ_p\right) \simeq 0$. 
\end{thm}

To begin, let $C$ be the completion of the algebraic closure of $\bQ_p$ and $\sO_C$ be its ring of integers which is an integral perfectoid ring. 

%\denis{Is there a reason why we use $C$ rather than the classical name $\mathbb{C}_p$?}

\begin{lem} \label{lem:tc} For $n \geq 2$, we have that $L_{T(n)}K\left(\sO_C\right) \simeq 0$. 
\end{lem}

\begin{proof} Consider the zig-zag of maps
\[
ku \rightarrow K\left(C\right) \xleftarrow{j^*} K\left(\sO_C\right).
\]
The maps above are all $p$-adic equivalences:
\begin{enumerate}
\item for $j^*$ this is \cite[Lemma 1.3.7]{hess-nik},
\item for the unlabeled arrow, we have Suslin rigidity \cite{Sus84}.
\end{enumerate}
Hence we conclude that $L_{T(n)}K\left(\sO_C\right) \simeq 0$. 
\end{proof}

\begin{rem} Note that \cite[Lemma 1.3.7]{hess-nik} only uses very basic facts about perfectoid rings (that we can choose a pseudouniformizer $\pi$ such that $\sO_C/\pi \simeq \sO_{C^{\flat}}/\pi^{\flat}$) and the fact that the positive homotopy groups of the $K$-theory of local perfect $\bF_p$-algebras are all $p$-divisible by \cite{hiller,kratzer}.
\end{rem}
%
%Recall that Hahn has resolved a conjecture of Hovey's \cite[Theorem 1.1]{Hahn16} to get that:
%    \begin{equation} \label{eq:hahn}
%     T\left(n\right) \otimes A = 0 \implies T\left(n+1\right) \otimes A = 0.
%     \end{equation}
%This lets us conclude that $L_{T(n)}K(\sO_C) \simeq 0$ for $n \geq 2$.

\begin{proof}[Proof of Theorem~\ref{thm:baby-mitchell}] Since $K$-theory is a finitary invariant\footnote{In more details: $\sO_C$ is $p$-adically isomorphic to the colimit of the $\sO_E$'s and $K$-theory preserves $p$-adic equivalences in this setting. This follows from, for example, the fact that given a morphism of rings $A \rightarrow B$, $\mrm{fib}(GL(A) \rightarrow GL(B)) \simeq \mrm{fib}(M(A) \rightarrow M(B))$ and the formation of matrix rings evidently preserves local equivalences.}, we can write
\[
K\left(\sO_C;\bZ_p\right) \simeq \colim_{\bQ_p \subset E \subset C} K\left(\sO_E;\bZ_p\right),
\]
a colimit of $\bE_{\infty}$-rings, and the colimit ranges along finite extensions of $\bQ_p$ contained in $C$. Now, the source vanishes after applying $L_{T\left(n\right)}$, whence so is the target. Since $L_{T\left(n\right)}$ commutes with filtered colimits, we may find a finite extension $E$ of $\bQ_p$ such that $L_{T(n)}K\left(\sO_E\right) \simeq 0$ since the colimit in sight is computed in $\bE_{\infty}$-ring; indeed, vanishing is equivalent to the unit being nullhomotopic. Since the morphism of rings $\bZ_p \rightarrow \sO_E$ is finite flat, by the descent results of \cite{CMNN} we get that
\[
L_{T\left(n\right)}K(\bZ_p) \simeq \mrm{Tot}\left(L_{T\left(n\right)}K\left( \sO_E^{\otimes_{\bZ_p} \bullet+1}\right)\right).
\]
We are now done, since all terms of the limit on the right hand side are modules over $L_{T\left(n\right)}K\left(\sO_E\right)=0$.
\end{proof}

To conclude note that if $A$ is $T(n)$-acyclic, then it is also $K(n)$-acyclic. Though we will not need it, we can also reverse the implication if $A$ is, furthermore, an $\bE_{\infty}$-ring spectrum by  \cite[Lemma 2.3]{land-meier-tamme}.

%Recall also: if $ A $ an $\bE_\infty $-ring, the following statements are equivalent:
%    \begin{enumerate}
%    \item $ K(n) \otimes A \simeq 0 $,
%    \item $ T(n) \otimes A = 0$. 
%    \end{enumerate}
%A reference for this statement may be found in \cite[Lemma 2.3]{land-meier-tamme}. Hence, we may conclude:
 
 \begin{cor}\label{cor:mitchell-kn} If $n \geq 2$, we have that $L_{K\left(n\right)}K\left(\bZ_p\right) \simeq 0$. In particular $K(n)_*K\left( \bZ_p\right) = 0$.
 \end{cor}

\ssec{Mitchell's theorem}
We now give a proof of Mitchell's theorem. We phrase this as. 

\begin{thm}\label{thm:mitchell-new} For all primes $p$ and $n \geq 2$, $L_{T(n)}K(\bZ) \simeq 0$. Equivalently, $L_{K(n)}K(\bZ) \simeq 0$ and, in particular, $K(n)_*K(\bZ) = 0$.
\end{thm}

\begin{proof}
    The ``equivalently" part follows from \cite[Lemma 2.3]{land-meier-tamme}. We first claim that that we can work with rings which are $\left(p\right)$-local. To see this, we claim that the map 
    \[
    L_{T\left(n\right)} K\left( \bZ \right) \rightarrow L_{T\left(n\right)}K\left(\bZ \left[\frac{1}{p} \right]\right),
    \]
    is an equivalence. 
    Indeed, by localization and d\'evissage \cite{quillen-k} we have a cofiber sequence
    \[
    K\left(\bF_p\right) \rightarrow K\left(\bZ\right) \rightarrow K\left(\bZ \left[\frac{1}{p} \right]\right),
    \]
    and $L_{T\left(n\right)}K(\bF_p) \simeq L_{T\left(n\right)}H\bZ_p \simeq 0$ where the first equivalence is \cite{quillen-finite} and the second equivalence follows since $n \geq 2$. Therefore, it suffices to show that $ {L_{T(n)}K\left(\bZ \left[\frac{1}{p} \right]\right) \simeq 0} $ which follows from the next two claims. 

%    
%     the proof of \cite[Theorem 2.16]{bhatt-clausen-mathew} furnishes a localization sequence
%    \[
%    \colim_s \Mod_{\bZ/p^s}(\Perf(\bZ)) \rightarrow \Perf\left(\bZ\right) \rightarrow \Perf\left(\bZ \left[\frac{1}{p} \right]\right).
%    \]
%    Since $K$-theory and $T(n)$-localizations preserve colimits, it suffices to prove that
%    \[
%    L_{T(2)}K(\bZ/p^s) \simeq 0.
%    \]
%   Since $K(\bZ/p^s)$ is an $\bE_{\infty}$-ring spectrum for all $s$, this follows from $L_{T(1)}K(\bZ/p^s) \simeq 0$ \cite[Proposition 2.14]{bhatt-clausen-mathew} and Hahn's result~\eqref{eq:hahn}. 

\end{proof}

\begin{claim} \label{claim1}
	For any $n \geq 1$, the presheaf, 
	\[
	 L_{T(n)}K(-):\Et^{\op}_{\bZ \left[\frac{1}{p} \right]} \rightarrow \Spt
	 \] is a hypercomplete sheaf of spectra. 
\end{claim}

\begin{proof} While this is an immediate application of \cite[Theorem 7.14]{clausen-mathew}, we will give a more detailed proof here to highlight the ingredients. Since telescopic localization is invariant under taking connective covers (\cite[Lemma 2.3(iii)]{land-meier-tamme}) we obtain a map of presheaves of $\bE_{\infty}$-rings:
\[K^{\mrm{cn}}\left(-;\bZ_p\right)\to L_{T\left(n\right)}K^{\mrm{cn}}\left(-;\bZ_p\right)\cong L_{T\left(n\right)}K\left(-\right)\]
Moreover $L_{T\left(n\right)}K(-)$ is an étale sheaf, thus this map factors through a canonical $\bE_{\infty}$-map
\[K^{\mrm{cn}}\left(-;\bZ_p\right)^{\et} \rightarrow L_{T\left(n\right)}K\left(-\right)\]
Since hypercompletion is smashing by \cite[Corollary 4.39]{clausen-mathew}, it suffices to then prove that $K^{\mrm{cn}}\left(-;\bZ_p\right)^{\et}$ is a hypersheaf on $\Et_{\bZ \left[\frac{1}{p} \right]}$. 

As is proved by Thomason in \cite[Theorem 4.1]{aktec} for odd primes (which relies on the Suslin-Merkerjuev theorem \cite{ms}) and Rosenschon and \O stv\ae r \cite{roso} for the prime $2$ (which does rely on the Milnor conjecture \cite{VV}), $L_{T(1)}K$ does satisfy \'etale hyperdescent. We consider the canonical map $K^{\mrm{cn}}\left(-;\bZ_p\right)^{\et} \rightarrow L_{T(1)}K$ and $\sF$ be the fiber. The claim then follows if one can show that the fiber $\sF$ has \'etale hyperdescent. 

Let $\sG$ denote the fiber of the map $K^{\mrm{cn}}(-;\bZ_p) \rightarrow L_{T(1)}K$. By the full strength of the Bloch-Kato conjectures \cite{Voevodsky:2008, VV} (see \cite[Theorem 6.18]{clausen-mathew}, noting that $\Spec \bZ $ admits the desired global bound by \cite[I.3.2]{Ser}) we see that the fiber is truncated. Therefore the sheafification, $\sF \simeq \sG^{\et}$ is Postnikov complete, whence is indeed hypercomplete as desired.

%The crucial part is the verification of the truncatedness hypothesis (3), which is a consequence of for the form that we need.\denis{I am a bit confused: don't we only care about the small étale site of $\bZ[1/p]$ anyway? Therefore the truncatedness would prove enough hyperdescent for our purposes...}
%

%
%the Clausen-Mathew theorem depends on the resolution of the Quillen-Lichtenbaum conjectures. This is not necessary for the generality of the present claim and we explain this. Indeed, since $L_{T(n)}K$ satisfies \'etale descent by \cite{CMNN}, there is a canonical map $K^{\mrm{cn}}(-;\bZ_p)^{\et} \rightarrow L_{T(n)}K$ under  $K^{\mrm{cn}}(-;\bZ_p)$. Since hypercompletion is smashing by \cite[Corollary 4.39]{clausen-mathew}, it suffices to then prove that $K^{\mrm{cn}}(-;\bZ_p)$ is a hypersheaf on $\Et_{\bZ \left[\frac{1}{p} \right]}$.

%When $p$ is odd, in this generality, the result follows from \cite[Theorem 4.1]{aktec}. 

\end{proof}

%\begin{proof}[Proof of Claim~\ref{claim1}] We have the following solid arrows in the $\infty$-category of Nisnevich sheaves of spectra on the small \'etale site of $ \Spec \bZ\left[\frac{1}{p} \right] $:
%\begin{equation*}
%\begin{tikzcd}[column sep=tiny,row sep=small]
%	K \ar[rr] \ar[rd] & & L_{T(n)}K  \\
%	& K^\et_{(\ell)} \ar[ru,dashed,"\exists"']
%\end{tikzcd} 
%\end{equation*}
%By \cite[Thm 1.3]{CMNN}, $ L_{T(n)}K $ is an {\'e}tale sheaf, whence the dotted arrow exists by the universal property of {\'e}tale sheafification. 
%Since the $ \ell $-virtual cohomological dimension of the residue fields of  we deduce from \cite[Thm 7.11]{CM19} that $ K^\et_{(\ell)} $ is a hypercomplete sheaf. 
%By \cite[Cor 4.39]{CM19} hypercompletion is smashing for ($\ell$-local) \'etale sheaves of spectra, hence $ L_{T(n)}K $ as a module (in fact an algebra) over a hypercomplete sheaf is itself hypercomplete and the claim follows from \cite[Thm 1.11]{CM19}.
%\end{proof}

\begin{claim}\label{claim2}
	For $n \geq 2$ and for all strictly hensel local ring $R$ with residue field $\kappa$ of characteristic $\ell > 0$ and $(p, \ell) = 1$
	\[
	L_{T\left( n \right)}K(R) \simeq 0
	\]  
\end{claim}

Indeed, since strictly henselian local rings are the points in the {\'e}tale topology and vanishing of \'etale hypersheaves are detected on points  \cite[Remark 6.5.4.7]{LurHTT}, the claim implies our result.

\begin{proof}[Proof of Claim~\ref{claim2}] By Gabber rigidity \cite{Gab92}, the map $K^{\mrm{cn}}(R) \rightarrow K^{\mrm{cn}}(\kappa)$ is a $p$-adic equivalence, while by Suslin rigidity \cite{Sus84} we have a further $p$-adic equivalence $K^{\mrm{cn}}(\kappa) \leftarrow \mrm{ku}$.  Since telescopic localization for $ n \geq 1 $ is invariant under taking connective covers by \cite[Lemma 2.3(iii)]{land-meier-tamme} we conclude:
\begin{equation*} 
	L_{T\left( n \right)}K \left(R \right) \simeq L_{T\left( n \right)}K^{\mrm{cn}}\left(R\right) \simeq L_{T\left( n \right)} \mrm{ku} \simeq 0.
\end{equation*} 
\end{proof}

%\begin{claim} 
%	$ L_{T(n)}K\left(\bZ_{(\ell)}^{sh} \right)  = 0 $ for all $ n \geq 2 $ and all primes $ \ell \neq p $.
%\end{claim}

\begin{rem}Because of the appeal to \cite[Theorem 7.14]{clausen-mathew}, our proof is not ``simple", in contrast to the $p$-adic situation. This is because the Clausen-Mathew theorem depends on the resolution of the Quillen-Lichtenbaum conjectures. Specifically, the version of  \cite[Theorem 7.14]{clausen-mathew} that we need, requires \cite[Theorem 6.13]{clausen-mathew} which ultimately proves that the restriction of $K^{\et}$ to $\Et_{\bZ[1/p]}$ is a \'etale hypersheaf. This latter result, in turn, relies on  Rost-Voevodsky's resolution of the Bloch-Kato conjectures. Note that, in contrast to the $p$-adic situation, Theorem~\ref{thm:mitchell-new} asserts something global which prompts us to argue via stalks, whence an appeal to hyperdescent.
\end{rem}

\begin{rem}[Personal communication by J. Rognes] Morally speaking, our proof is a cleaner packaging of the following method to prove Mitchell's theorem using the Rost-Voevodsky results.
For simplicity let $p$ be an odd prime and let $V(1):= \1/(p, v_1)$ be a finite complex which admits a $v_2$-self map. Our goal is to prove that $K(\bZ) \otimes V(1)$ is a bounded complex which suffices to prove Mitchell's theorem since inverting a positive degree self-map on a bounded complex annihilates it. Using the localization sequence, we reduce to proving the following assertions:
\begin{enumerate}
\item if $\ell \not= p$, then $v_1$ acts invertibly on $K_*(\bF_{\ell})/p$ for $*$ large enough,
\item $v_1$ acts invertibly on $K_*(\bQ)/p$ for $*$ large enough, and
\item if $\ell = p$, then $K_*(\bF_p)/p$ is bounded above.
\end{enumerate}
The last statement follows from Quillen's computation \cite{quillen-finite} and the first statement follows by a further application of Suslin rigidity \cite{Sus83}. The second statements is where Rost-Voevodsky's resolution of the Bloch-Kato conjectures is needed \cite{VV,Voevodsky:2008} (this is where the ``overkill happens"), the motivic spectral sequence \cite{f-sus} and the fact that $v_1$ is detected by a periodicity operator on \'etale cohomology. 

\end{rem}

As a consequence of the main theorem we can easily obtain two corollaries:

\begin{cor} The functor $L_{T(n)}K|_{\Cat^{\perf}_{\bZ}}$ is zero for $n \geq 2$.
\end{cor}

\begin{cor}\label{thm:mitchell-tc} For $n \geq 2$, $L_{T\left(n\right)}TC\left(\bZ\right) \simeq 0$. Consequently, $L_{T\left(n\right)}TC|_{\Cat^{\perf}_{\bZ}}$ is zero for $n \geq 2$.
\end{cor}

\begin{proof} Via the trace map, $L_{T\left(n\right)}TC(\bZ)$ is an $L_{T\left(n\right)}K\left(\bZ\right)$-$\bE_{\infty}$-algebra, whence the result follows from Theorem~\ref{thm:mitchell-new}.
\end{proof}

\bibliographystyle{alphamod}

\let\mathbb=\mathbf

{\small
\bibliography{references}
}

\parskip 0pt

\end{document}